\documentclass[11pt,reqno]{article}
\usepackage{graphicx}
\usepackage{amsmath}
\usepackage{amsthm}
\usepackage{latexsym,bm}
\usepackage{amssymb}
\usepackage{indentfirst}

\newtheorem{thm}{Theorem}[section]
\newtheorem{cor}[thm]{Corollary}
\newtheorem{lem}[thm]{Lemma}
\newtheorem{prop}[thm]{Proposition}
\newtheorem{rem}[thm]{Remark}

\newtheorem{Remark}{Remark}
\numberwithin{equation}{section}

\newcommand {\be}{\begin{equation}}
\newcommand {\ee}{\end{equation}}

\begin{document}

\title{Riemannian manifolds with entire  Grauert  tube are rationally elliptic}
\author{Xiaoyang Chen\footnotemark[1]}
\renewcommand{\thefootnote}{\fnsymbol{footnote}}
\footnotetext[1]{School of Mathematical Sciences, Institute for Advanced Study, Tongji University, Shanghai, China. email: $xychen100@tongji.edu.cn$.
}

\date{}

\maketitle

\begin{abstract}
It was conjectured by Bott-Grove-Halperin that a compact simply connected  Riemannian manifold $M$ with nonnegative sectional curvature is rationally elliptic.
We confirm this conjecture under the stronger assumption that $M$  has  entire Grauert tube, i.e.,
$M$ is a real analytic Riemannian manifold that has
 a unique adapted complex structure defined on the whole tangent bundle $TM$.

\end{abstract}

\section{Introduction}
The following conjecture formulated by Bott-Grove-Halperin  is a central problem in the study of Riemannian manifolds with nonnegative
sectional curvature \cite{Be, GH}.

Conjecture: A compact simply connected  Riemannian manifold $M$ with nonnegative sectional curvature is rationally elliptic.

Here $M$ is said to be rationally elliptic if and only if it has finite dimensional rational homotopy groups,  i.e., all but finitely many homotopy groups of $M$ are finite, otherwise $M$ is said to be rationally hyperbolic. It is a well-known simple consequence of Sullivan's minimal model theory
that $M$ being rationally elliptic is equivalent to polynomial growth of the sequence of Betti numbers of its based loop space $\Omega M$
relative to rational coefficient.
If $M$ is rationally elliptic, then there are severe topological restrictions of $M$. For example, $M$ has nonnegative Euler characteristic number and $ dim H_*(M, \mathbb{Q})\leq 2^n$ \cite{H, GH}.

\par It is known that compact simply connected homogeneous spaces and cohomogeneity one manifolds
 are rationally elliptic \cite{GH2}.
 In \cite{GB}, they confirmed  Bott-Grove-Halperin  conjecture under
 the additional assumption that $M$ supports an isometric action with orbits of codimension two.

 \par
In this paper we confirm  Bott-Grove-Halperin  conjecture under the stronger assumption that $M$ has entire  Grauert  tube:
\begin{thm} \label{main}.
Let $(M,g)$ be a $n$-dimensional compact  simply connected real analytic Riemannian manifold that has entire Grauert tube, then $M$ is rationally elliptic.
\end{thm}
\begin{rem}
In fact, our proof shows that $M$ is topologically elliptic, i.e. the Betti
numbers of its loop space relative to any field of coefficients grow at most
polynomially.
\end{rem}

 Here $(M,g)$ is said to be real analytic if $M$ is a real analytic manifold with a real analytic Riemannian metric $g$.
 Then there is a unique adapted complex structure defined on $T^R M=\{{v\in TM| g(v,v)<R^2}\}$ for some $R>0$ \cite{GS, Ls, Sz}.
 When $R=\infty$, then $M$ is said to have entire  Grauert  tube.
  It was shown in \cite{Ls} that a Riemannian manifold with  entire  Grauert  tube has nonnegative sectional curvature.
 Moreover,  Aguilar showed that the quotient of a Riemannian manifold with entire Grauert tube by a group of isometries acting freely also has entire Grauert tube \cite{Ag}.
  All known manifolds with entire Grauert tube are obtained by Aguilar's construction: starting with a compact Lie group with a bi-invariant metric, or the product of such a group with Euclidean space, one takes the quotient by some group of isometries acting freely. Such quotient manifolds include almost all closed manifolds which are known to have Riemannian metrics with nonnegative sectional curvature.

 \par  It was conjectured by Hopf that the Euler characteristic number of a compact Riemannian manifold with nonnegative
 sectional curvature  is nonnegative. The following corollary settles this conjecture under the stronger assumption that $M$ has entire  Grauert  tube.
 \begin{cor} \label{euler}
 Let $M$ be a compact Riemannian manifold with entire Grauert tube. Then $M$ has nonnegative Euler characteristic number.
 \end{cor}

\begin{proof}
If $M$ has finite fundamental group, then its universal cover $\widetilde{M}$ with the induced Riemannian metric also has entire Grauert tube.
By Theorem \ref{main}, the Euler characteristic number of $\widetilde{M}$ is nonnegative. Hence $M$ has  nonnegative Euler characteristic number.
If $M$ has infinite fundamental group, as $M$ has nonnegative sectional curvature, then the Euler characteristic number of $M$ is zero
\cite{CG}.
\end{proof}

 \par A related conjecture proposed by Totaro predicts that a compact Riemannian manifold $M$ with nonnegative sectional curvature has a good complexification,
i.e., $M$ is diffeomorphic to a smooth affine algebraic variety $U$ over the real number such that the inclusion $U(\mathbb{R})
\rightarrow U(\mathbb{C})$  is a homotopy equivalence. The Euler characteristic number of a compact manifold which has a good complexification
is also nonnegative \cite{To}. Also, a conjecture by Burns \cite{BL} predicts that for every compact Riemannian manifold $M$ with entire Grauert tube,
the complex manifold $TM$ is an affine algebraic variety in a natural way. If this is correct, the complex manifold $TM$ would be a good complexification
of $M$ in the above sense. Both conjectures of Totaro and Burns are still open.

\par The proof of Theorem \ref{main} is based on the counting function introduced in \cite{Be,Gr,P}.
 For $x \in M$ and each $T >0$, let
 $$D_T :=\{{ v \in T_x M| g(v,v)\leq T^2}\}$$
 be the disk of radius $T$ in $T_x M$. Define the counting function $n_T (x, y)$ by
 $$n_T (x, y):=\sharp ((exp_x)^{-1}(y)\cap D_T ).$$
 In other words, $n_T (x,y)$ counts the number of geodesic arcs joining $x$ to $y$ with length $\leq T$.
When $M$ is simply connected, then we have the following crucial inequality \cite{Gr,P}:
\begin{equation} \label{ine1}
\sum_{j=0}^{k-1} dim H_j (\Omega M, F) \leq \frac{1}{Vol_g (M)}\int_{M} n_{Ck} (x, y) dy,
\end{equation}
where $C$ is a positive constant independent of $k$ and $F$ is any field of coefficients.

\par For any $x \in M$, Berger and Bott proved that $\int_{M} n_{T} (x, y) dy$ can be computed by Jacobi fields on $M$ \cite{Be,P}. Precisely, they showed that
\begin{equation} \label{in1}
\int_M n_T (x, y) dy =\int_0^T  d \sigma \int_{\mathbb{S}} \sqrt{det (g(J_j (\sigma), J_k (\sigma)))_{j,k=1,2, \cdots, n-1}}d\theta,
\end{equation}
where $\mathbb{S}$ is the unit sphere of $T_x M$.  Moreover,
$J_j, j=1, 2, \cdots, n-1$ are Jacobi fields along the unique geodesic $\gamma$ determined by $\theta \in \mathbb{S}$ (i.e. $\gamma(0)=x, {\gamma}'(0)=\theta$) with initial conditions
$$J_j(0)=0$$
$$ {J}'_j (0)=v_j,$$
where $v_j, j=1, 2 \cdots n-1$ is an orthonormal basis of $T_{\theta} \mathbb{S}$.

\par If $(M, g)$ has entire  Grauert tube, the right hand side in \ref{in1} can be further described by a  matrix valued holomorphic function on the upper half plane.
Applying Fatou's representation theorem to this function, we will show that $\int_M n_T (x, y) dy $ is a polynomial  function of $T$. When $M$ is simply connected,
 it follows that $\sum_{j=0}^{k-1} dim H_j (\Omega M, F)$
has polynomial growth for any field of coefficients. Hence $M$ is topologically elliptic.

\par We finally mention that based on  an iterated use of the Rauch comparison theorem for Jacobi fields, an estimate for the Betti numbers of
 $\Omega M$ for manifolds with $0 < \delta  \leq  sec M \leq 1$ was derived in \cite{Be}. Although the estimate is given in terms of the pinching
 constant $\delta$, its growth rate is exponential.

\section*{Acknowledgements} The author is partially supported by National Natural Science Foundation of China No.11701427
and the Fundamental Research Funds for the Central Universities.

\section{Vertical and horizontal subbundles}
In this section we recall some basic facts on the geometry of the tangent bundle $TM$. For more details, see \cite{P}.
\par Let $\pi: TM \rightarrow M$ be the canonical projection, i.e., if $\theta=(x,v) \in TM$, then $\pi(\theta)=x$. There exists a canonical
subbundle of $TTM$ called the vertical subbbundle whose fiber at $\theta$ is given by the tangent vectors of curves $\sigma: (-\epsilon, \epsilon) \rightarrow
TM$ of the form: $\sigma(t)=(x, v+t \omega)$, where $\omega \in T_x M$. In other words,
$$V(\theta)=ker ((\pi_*)_{\theta}).$$

\par Suppose that $M$  is endowed with a Riemannian metric $g$. We shall define the connection map
$$K: TTM \rightarrow TM$$
as follows: let $\xi \in T_{\theta} TM$ and $z: (-\epsilon, \epsilon)\rightarrow TM$ be an adapted curve to $\xi$, that is, with initial conditions as follows:
$$z(0)=\theta$$
$$ {z}'(0)=\xi.$$
such a curve gives rise to a curve $\alpha: (-\epsilon, \epsilon) \rightarrow M, \alpha:=\pi \circ z$ and a vector field $Z$ along $\alpha$, equivalently,
$z(t)=(\alpha(t), Z(t)).$
\par Define
$$K_{\theta}(\xi):=(\nabla_{\alpha}Z) (0)=\lim_{t \rightarrow 0} \frac{(P_t)^{-1}Z(t)-Z(0)}{t},$$
where $P_t: T_x M \rightarrow T_{\alpha(t)} M$ is the linear isomorphism defined by the parallel transport along $\alpha$.
The horizontal subbundle is the subbundle of $TTM$ whose fiber at $\theta$ is given by
$$H(\theta)=ker K_{\theta}.$$
Another equivalent way of constructing the horizontal subbundle is by means of the horizontal lift
$$L_{\theta}: T_x M \rightarrow T_{\theta} TM,$$
which is defined as follows: let $\theta=(x, v)$. Given $\omega \in T_x M$ and $\alpha: (-\epsilon, \epsilon) \rightarrow M$ an adapted curve of $\omega$, i.e.,
$\alpha(0)=x, {\alpha}'(0)=\omega$. Let $Z(t)$ be the parallel transport of $v$ along $\alpha$ and $\sigma: (-\epsilon, \epsilon) \rightarrow TM$
be the curve $\sigma(t)=(\alpha(t), Z(t))$. Then
$$L_{\theta}(w)={\sigma}'(0)\in T_{\theta}TM.$$
\begin{prop}\label{jac}
$K_{\theta}$ and $L_{\theta}$ have the following properties:
$$ (\pi_*)_{\theta} \circ L_{\theta}=Id$$
$$K_{\theta} \circ i_*=Id,$$
where $i: T_x M \rightarrow TM $ is the inclusion map. Moreover,
$$T_{\theta}TM= H(\theta)\oplus V(\theta)$$
and the map $j_{\theta}: T_{\theta}TM  \rightarrow T_x M \times T_x M$ given by
$$j_{\theta} (\xi)=( (\pi_*)_{\theta} (\xi), K_{\theta}(\xi))$$
is a linear isomorphism.
\end{prop}

\par  For each $\theta \in TM$, there is a unique geodesic $\gamma_{\theta}$ in $M$ with initial condition $\theta$.
Let $\xi \in T_{\theta} TM$ and $z: (-\epsilon, \epsilon)\rightarrow TM$ be an adapted curve to $\xi$, that is, with initial conditions as follows:
$$z(0)=\theta$$
$$ {z}'(0)=\xi.$$
Then the map
$(s, t) \mapsto \pi \circ \phi_t (z(s))$ gives rise to a variation of $\gamma_{\theta}$. Here $\pi: TM \rightarrow M$ is the projection map
 and $\phi_t$ is the geodesic flow of $TM$. The curves $t \mapsto \pi \circ \phi_t (z(s))$ are geodesics
and therefore the corresponding variational vector fields $J_{\xi}:=\frac{\partial}{\partial s}|_{s=0} \pi \circ \phi_t (z(s))$ is a Jacobi field with initial conditions
$$J_{\xi}(0)=(\pi_*)_{\theta} (\xi)$$
$${J}'_{\xi}(0)=K_{\theta}({\xi}).$$

%Let us denote $J(\gamma_{\theta})$ the vector space of all Jacobi fields along $\gamma_{\theta}$. Then the map $i_{\theta}: T_{\theta} TM \rightarrow J(\gamma_{\theta})$ given by
%$$i_{\theta} (\xi)= J_{\xi}$$
%is a linear isomorphism. For this reason, we call $J_{\xi}$ the Jacobi field determined by $\xi$.

 \section{ Adapted complex structure on the tangent bundle}
 In this section we describe the adapted complex structure on the tangent bundle.
 Let $(M,g)$ be a compact smooth Riemannian manifold, then $TM\setminus M$ carries a natural foliation by Riemannian surfaces defined as follows:
 For $\tau \in \mathbb{R}$ denote by $N_{\tau}: TM \rightarrow TM $ the smooth mapping defined by multiplication by $\tau$ in the fibers. If $\gamma: \mathbb{R} \rightarrow M$
 is a geodesic, define an immersion $\phi_{\gamma}: \mathbb{C} \rightarrow TM$ by
 $$\phi_{\gamma}(\sigma+ i \tau)=N_{\tau}{\gamma}'(\sigma).$$
 If for two geodesics $\gamma,\delta$, $\phi_{\gamma}(\mathbb{C}\setminus \mathbb{R})$ and $\phi_{\delta}(\mathbb{C}\setminus \mathbb{R})$
 intersect each other, then $\gamma$ and $\delta$ are the same geodesic traversed with different velocities, hence $\phi_{\gamma}(\mathbb{C})=\phi_{\delta}(\mathbb{C}).$
 Therefore the images of $\mathbb{C}\setminus \mathbb{R}$ under the mapping $\phi_{\gamma}$ defines a smooth foliation of $TM \setminus M$ by surfaces. Moreover, each leaf has complex structure that it inherits from $\mathbb{C}$ via $\phi_{\gamma}$. The leaves, along with their complex structure extend across $M$, but of course, on $M$ the foliation
 $\mathcal{F}$
  becomes singular.
 \par

 Given $R>0$, put
 $$T^R M=\{{v\in TM| g(v,v)<R^2}\}.$$
 A smooth complex structure on $T^R M$ will be called adapted if the leaves of the foliation $\mathcal{F}$ with the complex structure inherited from $\mathbb{C}$
 are complex submanifolds of $T^R M$. In \cite{GS,Ls,Sz}, they proved the following

\begin{thm} Let $M$ be a compact real analytic manifold equipped with a real analytic metric $g$.
  Then there exists some $R>0$ such that $T^R M$ carries a unique adapted complex
 structure.
 \end{thm}
 When the adapted complex
 structure is defined on the whole tangent bundle, i.e. $R=\infty$, then $M$ is said to have entire  Grauert  tube.
  It was shown in \cite{Ls} that a Riemannian manifold with  entire  Grauert  tube has nonnegative sectional curvature.
 \par The adapted complex structure on $T^R M$ can be described as follows. For this purpose
 let $\theta \in T^R M \setminus M$ and $x=\pi(\theta)$, where $\pi: TM \rightarrow M$ is the projection map. Let
 $\gamma$ be a geodesic determined by $\theta$.  Choose tangent vectors $v_1, v_2, \cdots, v_{n-1}$ such that
 $v_1, v_2,  \cdots, v_{n-1}, v_n:=\frac{{\gamma}'(0)}{|{\gamma}'(0)|}$ form an  orthonormal basis of $T_x M$.

 \par Denote $L_{\theta}$ the leaf of the foliation $\mathcal{F}$ passing through $\theta$. A vector $\bar{\xi} \in T_{\theta} TM$ determines a vector field
 $\xi$ (we call it parallel vector field)
 along $L_{\theta}$ by defining it to be invariant under two semi-group actions. Namely $\xi$ is invariant under $N_{\tau}$ and the geodesic flow.
 For this parallel field $\xi$, we get that ${\xi}|_{\mathbb{R}}$ is a Jacobi field along $\gamma$.

 \par Now choose a set of vectors $\bar{\xi}_1, \bar{\xi}_2, \cdots, \bar{\xi}_{n}, \bar{\eta}_1, \bar{\eta}_2, \cdots, \bar{\eta}_{n} \in T_{\theta} TM$ satisfying
 $$ (\pi_*)_{\theta} (\bar {\xi}_j)=v_j, \ K_{\theta} (\bar {\xi}_j)=0$$
 $$ (\pi_*)_{\theta} (\bar {\eta}_j)=0, \ K_{\theta} (\bar {\eta}_j)=v_j.$$
 Here $K: TTM \rightarrow TM$ is the connection map described in section $2$. Extend $\bar {\xi}_j$ and $\bar {\eta}_j$ to get parallel vector fields
 $\xi_1, \xi_2, \cdots, \xi_n, \eta_1, \eta_2, \cdots, \eta_n$ along $L_{\theta}$.
Then the Jacobi fields $\xi_1|_{\mathbb{R}}, \xi_2|_{\mathbb{R}}, \cdots, \xi_n|_{\mathbb{R}}$ are  linearly independent except on a discrete subset $S_1$ of $\mathbb{R}$.
 Hence there are smooth real valued functions $\phi_{jk}$ defined on $\mathbb{R} \setminus S_1$ such that
 $$\eta_k|_{\mathbb{R}}=\sum_{j=1}^n \phi_{jk} \xi_j|_{\mathbb{R}}.$$
 From the presence of the adapted complex structure it follows that the functions $\phi_{jk}$ have meromorphic extension $f_{jk}$ over the domain
 $$D=\{{\sigma + i \tau \in \mathbb{C}| \ |\tau|< \frac{R}{\sqrt{g(\theta, \theta)}}}\}$$
 such that for each $j, k$, the poles of $f_{jk}$ lies on $\mathbb{R}$ and the matrix $Im (f_{jk})|_{D\setminus \mathbb{R}}$ is invertible.
 Let $(e_{jk})=(Im f_{jk}(i))^{-1}$. Then  the complex structure $J$ satisfies
 $$J \bar{\xi_h}=\sum_{k=1}^{n} e_{kh}\times
  [ \bar{\eta_k}-\sum_{j=1}^{n} Re  f_{jk}(i) \bar{\xi_j}].$$

 \begin{Remark} \label{rem}

  Because $\xi_1|{\mathbb{R}}, \xi_2|{\mathbb{R}}, \cdots, \xi_{n-1}|{\mathbb{R}}, \eta_1|{\mathbb{R}}, \eta_2|{\mathbb{R}}, \cdots,
  \eta_{n-1}|{\mathbb{R}}$ are normal Jacobi fields, while $\xi_n|{\mathbb{R}}, \eta_n|{\mathbb{R}}$ are tangential Jacobi fields, for $ \ 1 \leq j, k \leq n-1$,
  we have
  $$\phi_{nk}=\phi_{jn}\equiv 0$$
  $$f_{nk}=f_{jn}\equiv 0$$
  $$e_{nk}=e_{jn}\equiv 0$$
  %$$\phi_{nn}(\sigma)=\sigma$$
  %$$f_{nn}(\sigma+ i \tau)=\sigma+ i \tau$$
  %$$e_{nn}(\sigma+ i \tau)=\frac{1}{\sigma+ i \tau}.$$
\end{Remark}

 Consider the $n$-tuples
 $$\Xi=(\xi_1, \xi_2, \cdots, \xi_n), \ H=(\eta_1, \eta_2, \cdots, \eta_n)$$
 and holomorphic $n$-tuples
 $$\Xi^{1,0}=(\xi_1^{1,0}, \xi_2^{1,0}, \cdots, \xi_n^{1,0}),  \ H^{1,0}=(\eta_1^{1,0}, \eta_2^{1,0}, \cdots, \eta_n^{1,0}),$$
 where $\xi_j^{1,0}=\frac{1}{2}(\xi_j- i J \xi_j)$ and $J$ is the adapted complex structure.

 \par Then we have
 $$H(\sigma)=\Xi(\sigma) f(\sigma)$$
 $$H^{1,0}(\sigma+ i \tau)=\Xi^{1,0} (\sigma+ i \tau) f(\sigma + i \tau)$$
  $$f(\sigma+ i \tau)=(f_{jk}(\sigma+ i \tau)), \ \sigma \in \mathbb{R}\setminus S_1, \ |\tau| < \frac{R}{\sqrt{g(\theta, \theta)}}.$$
\par The following facts are proved in \cite{Ls,Sz}.

\begin{prop} \label{linear}
(1) The vectors $\xi_1^{1,0}, \xi_2^{1,0}, \cdots \xi_n^{1,0}$ are linearly independent over $\mathbb{C}$ on $D \setminus \mathbb{R}$.
The same is true for the vectors $\eta_1^{1,0}, \eta_2^{1,0}, \cdots, \eta_n^{1,0}$.
\\
(2) The $2n$ vectors $\xi_j, \eta_k$ are linearly independent in points $\sigma+ i \tau \in D \setminus \mathbb{R}$.
\end{prop}

\begin{thm} \label {nice}
The matrix valued meromorphic functions $f(\sigma + i \tau )$ is symmetric (as a matrix) and satisfies

$$ f(0)=0, f'(0)=Id.$$
Moreover, if $\sigma + i \tau \in D, \tau >0$, then $Im f(\sigma + i \tau)$ is a symmetric, positive definite matrix.
\end{thm}

\section{Growth rate of counting functions}
In this section we prove Theorem \ref{main}.
\par
 Let $M$ be a $n$-dimensional compact manifold endowed with a Riemannian metric $g$. For $x \in M$ and each $T >0$, let
 $$D_T :=\{{ v \in T_x M|g(v,v)\leq T^2}\}$$
 be the disk of radius $T$ in $T_x M$. Define the counting function $n_T (x, y)$ by
 $$n_T (x, y):=\sharp ((exp_x)^{-1}(y)\cap D_T ).$$
 In other words, $n_T (x,y)$ counts the number of geodesic arcs joining $x$ to $y$ with length $\leq T$.

\par
 The following Theorems proved in \cite{Be, Gr,P} will be crucial for us.
\begin{thm} \label{gromov0}
\begin{equation} \label{ine2}
\int_M n_T (x, y) dy =\int_0^T  d \sigma \int_{\mathbb{S}} \sqrt{det (g(J_j (\sigma), J_k (\sigma)))_{j, k=1,2, \cdots, n-1}} \ d\theta,
\end{equation}
where $\mathbb{S}$ is the unit sphere of $T_x M$.  Moreover,
$J_j, j=1, 2, \cdots, n-1$ are Jacobi fields along the unique geodesic $\gamma$ determined by $\theta \in \mathbb{S}$ (i.e. $\gamma(0)=x, {\gamma}'(0)=\theta$) with initial conditions
$$J_j(0)=0$$
$$ {J}'_j (0)=v_j,$$
where $v_j, j=1, 2 \cdots n-1$ is an orthonormal basis of $T_{\theta} \mathbb{S}$.

\end{thm}

\begin{thm} \label{gro}
Let $M$ be a $n$-dimensional compact simply connected manifold endowed with a Riemannian metric $g$, then
\begin{equation} \label{gromov}
\sum_{j=0}^{k-1} dim H_j (\Omega M, F) \leq \frac{1}{Vol_g (M)}\int_{M} n_{Ck} (x, y) dy
\end{equation}
where $C$ is a positive constant independent of $k$ and $F$ is any field of coefficients.
\end{thm}

\begin{rem}
The assumption that $M$ is simply connected in Theorem \ref{gro} is essential.
\end{rem}

\par When $M$ has  entire Grauert tube, we will see that the right hand side in \ref{gromov0} can be further described by
a  matrix valued holomorphic function on the upper half plane. Applying Fatou's representation theorem to
this function, we will derive that $\int_M n_T (x, y) dy$ has  polynomial growth  and hence $M$ is topologically elliptic.

\par  Now we give the details of the proof. Let $\mathbb{S}$ be the unit sphere of $T_x M$ and
$\gamma$ the unique geodesic determined by
$\theta \in \mathbb{S}$, i.e. $\gamma(0)=x, {\gamma}'(0)=\theta$.
Let $v_1, v_2, \cdots, v_n:={\gamma}'(0)$ be an orthonormal basis of $T_x M$.

 \par As in section 3, choose a set of vectors $\bar{\xi}_1, \bar{\xi}_2, \cdots, \bar{\xi}_{n}, \bar{\eta}_1, \bar{\eta}_2, \cdots, \bar{\eta}_{n} \in T_{\theta} TM$ satisfying
 $$ \pi_* (\bar {\xi}_j)=v_j, \ K \bar {\xi}_j=0$$
 $$ \pi_* (\bar {\eta}_j)=0, \ K \bar {\eta}_j=v_j.$$
 Here $K: TTM \rightarrow TM$ is the connection map described in section $2$. Extend $\bar {\xi}_j$ and $\bar {\eta}_j$ to get parallel vector fields
 $\xi_1, \xi_2, \cdots, \xi_n, \eta_1, \eta_2, \cdots, \eta_n$.
 Then   $J_j:=\eta_j |{\mathbb{R}}, j=1, 2, \cdots, n-1$ are normal Jacobi fields along $\gamma$ with initial conditions
$$J_j(0)=0, \ {J}'_j (0)=v_j.$$
 Moreover, $\xi_1|_{\mathbb{R}}, \xi_2|_{\mathbb{R}}, \cdots, \xi_n|_{\mathbb{R}}$ are linearly independent except on a discrete subset $S_1$ of $\mathbb{R}$.
 Hence there are smooth real valued functions $\phi_{jk}$ defined on $\mathbb{R} \setminus S_1$ such that
 $$\eta_k|_{\mathbb{R}}=\sum_{j=1}^n \phi_{jk} \xi_j|_{\mathbb{R}}.$$
 As $M$ has  entire Grauert tube,  it follows that the functions $\phi_{jk}$ have meromorphic extension $f_{jk}$ over the whole complex plane
 such that for each $j, k$, the poles of $f_{jk}$ lies on $\mathbb{R}$ and the matrix $Im (f_{jk})|_{\mathbb{C} \setminus \mathbb{R}}$ is invertible.

\par Consider the $n$-tuples
 $$\Xi=(\xi_1, \xi_2, \cdots, \xi_n), \ H=(\eta_1, \eta_2, \cdots, \eta_n)$$
 and holomorphic $n$-tuples
 $$\Xi^{1,0}=(\xi_1^{1,0}, \xi_2^{1,0}, \cdots, \xi_n^{1,0}),  \ H^{1,0}=(\eta_1^{1,0}, \eta_2^{1,0}, \cdots, \eta_n^{1,0}),$$
 where $\xi_j^{1,0}=\frac{1}{2}(\xi_j- i J \xi_j)$ and $J$ is the adapted complex structure.

 \par Then we have
 $$H(\sigma)=\Xi(\sigma) f(\sigma)$$
 $$H^{1,0}(\sigma+ i \tau)=\Xi^{1,0} (\sigma+ i \tau) f(\sigma + i \tau)$$
  $$f(\sigma+ i \tau)=(f_{jk}(\sigma+ i \tau)), \ \sigma \in \mathbb{R}\setminus S_1.$$
\par
 For $\sigma \in \mathbb{R}\setminus S_1$, we can view $\Xi(\sigma), H(\sigma)$ as linear mappings $\mathbb{R}^n \rightarrow T_{\gamma(\sigma)} M$ given by
$$(\omega_j)=\omega \mapsto \Xi(\sigma) \omega=\sum_{j=1}^n \omega_j \xi_j(\sigma)$$
and similarly for $H(\sigma)$. Denote $\Xi^*(\sigma), H^*(\sigma)$ the adjoint of $\Xi(\sigma), H(\sigma)$, respectively (adjoint defined using the
Euclidean scalar product on $\mathbb{R}^n$ and the Riemannian metric on $T_{\gamma (\sigma)}M$). By the proof of Proposition 6.11 in \cite{Ls}, we get
$$\Xi^*(\sigma)\Xi(\sigma)f'(\sigma)=Id, \ \sigma \in (0, c)$$
for some positive constant $c$. Let $e_j$ be the standard orthonormal  basis of $\mathbb{R}^n$, then
$$\Xi(\sigma)e_j =\xi_j(\sigma)$$
$$\Xi^*(\sigma)\Xi(\sigma)e_j =g(\xi_j(\sigma), \xi_k(\sigma))e_k.$$
Hence $\Xi^*(\sigma)\Xi(\sigma)$ is real analytic over $\mathbb{R}$. By analytic continuation, we have
$$\Xi^*(\sigma)\Xi(\sigma)f'(\sigma)=Id$$
for every $\sigma \in \mathbb{R}\setminus S_1$.

\begin{lem} \label{nice0}
If $\sigma+ i \tau \in \mathbb{C} \setminus \mathbb{R}$, then $Im f^{-1}(\sigma + i \tau)$ is invertible.
\end{lem}

\begin{proof}
The proof is almost identical to the proof of Proposition 6.8 in \cite{Ls}. Suppose there is a nonzero column vector $v=(v_j) \in \mathbb{R}^n$ such that
$Im f^{-1}(\sigma + i \tau) v=0, \tau \neq 0$, i.e., $\omega=(\omega_k)=f^{-1}(\sigma+ i \tau) v \in \mathbb{R}^n$.
By Proposition \ref{linear}, $f^{-1}$ exists on $\mathbb{C}\setminus \mathbb{R}$. Then we have
$$\Xi^{1,0}= H^{1,0} f^{-1}$$
in the point $\sigma + i \tau$. Hence
$$\sum \xi_j^{1,0} v_j=\Xi^{1,0} v=H^{1,0} f^{-1}v=H^{1,0} \omega=\sum \eta_k^{1,0} \omega_k.$$
Taking real parts, we get
$$\sum \xi_j v_j=\sum \eta_k \omega_k,$$
in contradiction with Proposition \ref{linear}.

\end{proof}

\begin{lem} \label{key9}
$G(\zeta):=-f^{-1}(\zeta)$ is a matrix valued  meromorphic function on $\mathbb{C}$ whose pole lies in a discrete subset of $\mathbb{R}$
and $Im G(\zeta)$ is positive definite for $\zeta =\sigma+ i \tau \in \mathbb{C}^+$, where $\mathbb{C}^+$ is the upper half plane.
\end{lem}

\begin{proof}
Since $H^{1,0}$ and $\Xi^{1,0}$ are invertible on $\mathbb{C}$ except a discrete subset, combined with $H^{1,0}=\Xi^{1,0} f $,
we get that $G(\zeta)$ is a matrix valued  meromorphic function on $\mathbb{C}$ whose pole lies in a discrete subset of $\mathbb{R}$.
By Theorem \ref{nice}, we have $$f(0)=0, f'(0)=Id.$$
Then for small positive $\tau$, we get
$$Im \ G(i \tau)=Im (-f^{-1}(i \tau))$$
$$=Im \ (-(f(0)+ i \tau f'(0) + O(\tau^2))^{-1})$$
$$=Im \ (-i \tau Id +  O(\tau^2))^{-1}$$
$$=Im \ (\frac{i}{\tau} (Id + O(\tau)^{-1})).$$
Hence $Im \ G(i \tau)$ is positive definite for small positive $\tau$.  As
 $Im \ G(\zeta)$ is nondegenerate on $\mathbb{C}^+$ by Lemma \ref{nice0}, therefore
 $Im \ G(\zeta)$ is positive definite for $\zeta =\sigma+ i \tau \in \mathbb{C}^+$.

\end{proof}

Let $f_1=(f_{jk}), \ j, k=1, 2, \cdots, n-1.$ Then we have

\begin{lem} \label{key1}
There exists a discrete subset $S_2 \subset \mathbb{R}$ such that for $\sigma \in \mathbb{R}\setminus  S_2$, we have
\begin{equation} \label{ine2}
det(g(J_j (\sigma), J_k (\sigma))_{j, k=1,2, \cdots, n-1}=\frac{1}{det ((-f_1^{-1})'(\sigma))},
\end{equation}
 where $J_j, j=1, 2, \cdots, n$ are normal Jacobi fields along $\gamma$ with initial conditions
 $$J_j(0)=0, \ {J}'_j (0)=v_j$$
 and $v_1, v_2, \cdots, v_n:={\gamma}'(0)$ is an orthonormal basis of $T_x M$.

\end{lem}

\begin{proof}
Let $e_j$ be the standard orthonormal  basis of $\mathbb{R}^n$. As $f^{-1}(\sigma)$ exists on $\sigma \in \mathbb{R}\setminus  S_2$ for some discrete subset $S_2$,
then we get
$$g(J_j(\sigma), J_k (\sigma))=g(H(\sigma)e_j, H(\sigma)e_k)=\langle H^*(\sigma)H(\sigma)e_j, e_k \rangle.$$
Since $f(\sigma)$ is symmetric by Theorem \ref{nice}, combined with $H(\sigma)=\Xi(\sigma) f(\sigma)$ and $\Xi^*(\sigma)\Xi(\sigma)f'(\sigma)=Id$, we get
$$H^*(\sigma)H(\sigma)=(\Xi(\sigma)f(\sigma))^*\Xi(\sigma)f(\sigma)$$
$$=f(\sigma)\Xi^*(\sigma)\Xi(\sigma)f(\sigma)$$
$$=f(\sigma)(f'(\sigma))^{-1}f(\sigma)$$
$$=((-f^{-1})'(\sigma))^{-1}.$$

Since $f_{jn}=f_{nk}=0, \ j, k=1, 2, \cdots, n-1$, then we see that
$$det(g(J_j (\sigma), J_k (\sigma))_{j, k=1,2, \cdots, n-1}=\frac{1}{det ((-f_1^{-1})'(\sigma))}, \ \sigma \in \mathbb{R} \setminus S_2. $$
\end{proof}

The following  Fatou's representation theorem will be crucial for us.

\begin{prop} \label{fatou}
Let $F$ be an $n \times n$ matrix valued holomorphic function on the upper half plane $\mathbb{C}^+=\{{\xi \in \mathbb{C}| Im \ \zeta >0}\} \cup
(\mathbb{R} \setminus P)$, where $P$ is a discrete subset of $\mathbb{R}$ consisting of poles of $F$.
  Suppose that for every $\zeta \in \mathbb{C}^+, Im F(\zeta)$ is a symmetric, positive definite matrix, whereas for $\zeta \in \mathbb{R} \setminus P$, $Im F(\zeta)=0$.
Then there is an $n \times n$ symmetric matrix $\mu=(\mu_{jk})$ whose entries are real valued, signed Borel measures on $\mathbb{R}$ such that
\par $1^{\circ}$  $\mu_{jk}$ does not have mass on any interval which does not contain a pole of $F$;
\par $2^{\circ}$  $\int_{-\infty}^{+\infty} \frac{|d\mu_{jk}(t)|}{1+t^2}< \infty$;
\par $3^{\circ}$  $\mu$ is positive semidefinite in the sense that for any $(\omega_j) \in \mathbb{R}^n$, the measure $\sum\omega_j \omega_k \mu_{jk}$ is nonnegative;
\par $4^{\circ}$   $F'(\zeta)=A + \frac{1}{\pi} \int_{-\infty}^{+\infty} \frac{d\mu(t)}{(\zeta-t)^2}, \ \zeta \in \mathbb{C}^+$,
\par where $A$ is a symmetric, positive semidefinite constant matrix. In fact, we have
$A=\lim_{\tau \rightarrow +\infty} \frac{Im F(i\tau)}{\tau}$
and $d\mu(\sigma)$ is the weak limit of $Im F(\sigma+ i\tau)$ as $\tau \rightarrow 0^+$.
\end{prop}

\begin{proof}
See \cite{K} and Proposition 7.4 in \cite{Ls}. The only difference is that we require $F$ has a holomorphic extension to $\mathbb{R} \setminus P$,
hence we get that  $\mu_{jk}$ does not have mass on any interval which does not contain a pole of $F$.
\end{proof}

Now we are going to finish the proof of Theorem \ref{main}. Applying Proposition \ref{fatou} to the matrix valued holomorphic function $(-f_1^{-1})$ on the upper half plane, we get
\begin{equation} \label{int}
 (-f_1^{-1})'(\zeta)=A + \frac{1}{\pi} \int_{-\infty}^{+\infty} \frac{d\mu(t)}{(\zeta-t)^2}, \ \zeta \in \mathbb{C}^+,
\end{equation}
where $A=(a_{jk})$ is a symmetric, positive semidefinite constant matrix and $\mu$ is a $n \times n$ positive semidefinite symmetric matrix  whose entries are real valued, signed Borel measures on $\mathbb{R}$.
 By analytic continuation, equation \ref{int} also holds on $ \mathbb{R}$ except a discrete subset. Moreover,
  $\mu$ does not have mass on any interval  which does not contain a pole of $-f_1^{-1}$. This yields that
 $$(-f_1^{-1})'(\sigma)=A + \frac{1}{\pi} \sum_j \frac{\mu(t_j)}{(\sigma-t_j)^2},  \ \sigma \in \mathbb{R} \setminus \{{t_1, t_2, \cdots}\},$$
 where $\{{t_1, t_2, \cdots}\}$ are poles of $-f_1^{-1}$. As $f(0)=0$, we see that $0$ is pole of $-f_1^{-1}.$
\par
\begin{lem} \label{pi}
$$\mu(0)=\pi Id$$
\end{lem}
\begin{proof}
By Proposition \ref{fatou}, we get
$$\mu(0)=\lim_{\delta\rightarrow 0^+} \mu(-\delta, \delta)=\lim_{\delta\rightarrow 0^+} \lim_{\tau \rightarrow 0^+} \int_{-\delta}^{\delta} Im (-f_1^{-1}(\sigma+ i\tau)) d\sigma$$
$$=\lim_{\delta\rightarrow 0^+} \lim_{\tau \rightarrow 0^+} \int_{-\delta}^{\delta} Im (-(f_{jk}(0)+ f'_{jk}(0)(\sigma+ i \tau) + O(\sigma + i \tau)^2)^{-1}_{1 \leq j,k \leq n-1}) d\sigma$$
$$=\lim_{\delta\rightarrow 0^+} \lim_{\tau \rightarrow 0^+} \int_{-\delta}^{\delta} Im (-((\sigma+ i \tau)Id + O(\sigma + i \tau)^2)^{-1}) d\sigma$$
$$=\lim_{\delta\rightarrow 0^+} \lim_{\tau \rightarrow 0^+} \int_{-\delta}^{\delta} Im (-\frac{1}{\sigma+ i \tau} (Id+ O(\sigma + i \tau))^{-1}) d\sigma$$
$$=\lim_{\delta\rightarrow 0^+} \lim_{\tau \rightarrow 0^+} \int_{-\delta}^{\delta} Im (-\frac{1}{\sigma+ i \tau} Id+ O(1)) d\sigma$$
$$=\lim_{\delta\rightarrow 0^+} \lim_{\tau \rightarrow 0^+} \int_{-\delta}^{\delta} \frac{\tau}{\sigma^2+ \tau^2}  d\sigma \ Id $$
$$=\pi \ Id.$$

\end{proof}
Given Lemma \ref{pi}, then we have
$$(-f_1^{-1})'(\sigma)=\frac{1}{{\sigma}^2} Id + B, $$
where $B=A + \frac{1}{\pi} \sum_{t_j \neq 0} \frac{\mu(t_j)}{(\sigma-t_j)^2}$ is positive semidefinite.

\begin{lem} \label{min200}
Let $ A_1, A_2$ be two $k\times k$ Hermitian positive semidefinite complex matrix, then
$$det (A_1 +A_2) \geq det A_1 + det A_2.$$
\end{lem}

\begin{proof}
It follows from the Minkowski determinant theorem (page 115, \cite{ma}):
$$(det (A_1 + A_2))^{1/k} \geq (det A_1)^{1/k} + (det A_2)^{1/k}.$$
\end{proof}
By Theorem \ref{nice}, we get that $f(\sigma+ i \tau)$ is a symmetric matrix, so is $-f_1^{-1}(\sigma+ i \tau)$.
By Proposition \ref{fatou}, we see that $A$ and $\mu(t_j)$ are real valued symmetric positive semidefinite matrix. By Lemma \ref{min200}, we get
$$\frac{1}{det  ((-f_1^{-1})'(\sigma))} \leq  {\sigma}^{2n-2}.$$

\par By Theorem \ref{gromov0} and Lemma \ref{key1}, we see
$$\int_M n_T (x, y) dy \leq p(T), $$
where $p(T)$ is a polynomial of degree at most $n$. By Theorem \ref{gro}, $\sum_{j=0}^{k-1} dim H_j (\Omega M, F)$ has polynomial growth
for any field of coefficients. It follows that
 $M$ is topologically elliptic.

 \par To illustrate the idea of the above proof, we give two examples here. Let $M$ be a $n$-dimensional compact manifold of constant sectional curvature $c$.
 From the proof of Theorem 2.5 in \cite{Sz},  we have
  $$ f_1(\sigma+ i \tau)=(\sigma + i \tau) \ Id,  \ c=0$$
  $$ f_1(\sigma+ i \tau)= (tg (\sigma+ i \tau)) \ Id, \ c=1.$$

\par Case 1: When $c=0$, then $-f_1^{-1}(\sigma+ i \tau)=(-\frac{1}{\sigma+ i \tau}) \ Id$.  Hence
$$(-f_1^{-1})'(\sigma)=\frac{1}{\sigma^2} \ Id.$$
 Let $F(\sigma+ i \tau):=-f_1^{-1}(\sigma+ i \tau).$
 In this case, the matrix $A$ and measure $\mu$ in Proposition \ref{fatou} can be computed by
$$A=\lim_{\tau \rightarrow +\infty} \frac{Im F(i\tau)}{\tau}=0$$
$$\mu(0)=\lim_{\delta\rightarrow 0^+} \mu(-\delta, \delta)=\lim_{\delta\rightarrow 0^+} \lim_{\tau \rightarrow 0^+} \int_{-\delta}^{\delta} Im F(\sigma+ i\tau)d\sigma=\pi \ Id.$$
Then $\int_M n_T (x, y) dy $ has polynomial growth of degree $n$.
\par Case 2: When $c=1$, then $-f_1^{-1}(\sigma+ i \tau)=(-ctg (\sigma+ i \tau)) \ Id$.  Hence
$$(-f_1^{-1})'(\sigma)=\frac{1}{sin^2 (\sigma)} \ Id.$$
Let $F(\sigma+ i \tau):=-f_1^{-1}(\sigma+ i \tau).$ In this case, the matrix $A$ and measure $\mu$ in Proposition \ref{fatou} can be computed by
$$A=\lim_{\tau \rightarrow +\infty} \frac{Im F(i\tau)}{\tau}=0$$
$$\mu(j\pi)\equiv \mu(0)=\lim_{\delta\rightarrow 0^+} \mu(-\delta, \delta)=\lim_{\delta\rightarrow 0^+} \lim_{\tau \rightarrow 0^+} \int_{-\delta}^{\delta} Im F(\sigma+ i\tau)d\sigma=\pi \ Id, \ j \in \mathbb{Z}.$$
Then $\int_M n_T (x, y) dy $ has linear growth.

\end{document}